\documentclass[a4paper]{amsart}

\usepackage{amssymb}
\usepackage{latexsym}
\usepackage{amsmath}
\usepackage{euscript}

\def\sD{{\mathfrak D}}

      \def\dC{{\mathbb C}}

   \def\dN{{\mathbb N}}   
      \def\dR{{\mathbb R}}

   \def\dZ{{\mathbb Z}}

\DeclareMathOperator\dom{dom}
\DeclareMathOperator\ran{ran}

\DeclareMathOperator\Imag{Im}

\renewcommand\Im{\Imag}

\def\bm\chi{\mbox{\boldmath$\chi$}}

\def\sgn{{\rm sgn\,}}

\def\sgn{{\text{\rm sgn\,}}}

\def\loc{{\text{\rm loc\,}}}
\def\min{{\text{\rm min\,}}}
\def\max{{\text{\rm max\,}}}
\def\ess{{\text{\rm ess}}}

\def\senki{{\lbrack\negthinspace [\bot ]\negthinspace\rbrack}}
\def\senki+{{\lbrack\negthinspace [+] \negthinspace\rbrack}}

\newcounter{counter_a}

\newtheorem{theorem}{Theorem}[section]
\newtheorem{proposition}[theorem]{Proposition}
\newtheorem{corollary}[theorem]{Corollary}
\newtheorem{lemma}[theorem]{Lemma}

\theoremstyle{definition}

\numberwithin{equation}{section}

\begin{document}

\title[Singular left-definite Sturm-Liouville operators]
{Eigenvalue estimates for singular left-definite Sturm-Liouville operators}

\author{Jussi Behrndt}
\address{Institut f\"{u}r Numerische Mathematik,
Technische Universit\"{a}t Graz,
Steyrergasse 30,
A-8010 Graz, Austria}
\email{behrndt@tugraz.at}

\author{Roland M\"{o}ws}
\address{Institut f\"{u}r Mathematik, Fakult\"{a}t f\"{u}r Mathematik und Naturwissenschaften,
Technische Universit\"{a}t Ilmenau, Postfach 10 05 65, D-98684 Ilmenau, Germany}
\email{roland.moews@tu-ilmenau.de}

\author{Carsten Trunk}
\address{Institut f\"{u}r Mathematik, Fakult\"{a}t f\"{u}r Mathematik und Naturwissenschaften,
Technische Universit\"{a}t Ilmenau, Postfach 10 05 65, D-98684 Ilmenau, Germany}
\email{carsten.trunk@tu-ilmenau.de}


\begin{abstract}
The spectral properties of a singular left-definite Sturm-Liouville operator $JA$ are investigated and
described via the properties of the corresponding right-definite selfadjoint counterpart $A$ which is obtained by substituting the
indefinite weight function by its absolute value. The spectrum of the $J$-selfadjoint operator $JA$ is real and it follows that an interval $(a,b)\subset\dR^+$ is
a gap in the essential spectrum of  $A$
if and only if both intervals $(-b,-a)$ and $(a,b)$ are gaps in the essential spectrum of the
$J$-selfadjoint operator $JA$.  As one of the main results it is shown that the number of eigenvalues of $JA$ in
$(-b,-a) \cup (a,b)$  differs at most by three
of the number of eigenvalues of $A$ in the gap $(a,b)$; as a byproduct results on the accumulation
of eigenvalues of singular left-definite
Sturm-Liouville operators are obtained. Furthermore,
left-definite problems with symmetric and periodic coefficients are treated, and several examples are included to
illustrate the general results.
\end{abstract}

\maketitle

\section{Introduction}
We investigate spectral properties of a
Sturm-Liouville differential operator
associated with the differential expression
\begin{equation}\label{tau2222}
\tau =\frac{1}{r}\left(-\frac{d}{dx}\,p\,\frac{d}{dx}+q\right),\qquad r,p^{-1},q\in L^1_\loc(\dR)\,\,\,\text{real},
\,\,\, p>0 \,\,\,\mbox{a.e.}
\end{equation}
In contrast to standard Sturm-Liouville theory we do not assume that the weight function $r$ is positive.
Instead we consider {\it indefinite} Sturm-Liouville operators and differential expressions; here it will
be assumed that there exists some $c\in\dR$ such that the weight function $r$ is positive on $(c,\infty)$
and negative on $(-\infty,c)$.
Suppose that the corresponding {\it definite} differential expression
\begin{equation}\label{ell2222}
\ell =\frac{1}{\vert r\vert}\left(-\frac{d}{dx}\,p\,\frac{d}{dx}+q\right)
\end{equation}
is in the limit point case at both singular endpoints $-\infty$ and $\infty$, or, equivalently, that the maximal
differential operator $A$ associated with $\ell$ in the weighted Hilbert space $L^2_{\vert r\vert}(\dR)$
is selfadjoint. If $J$ denotes the multiplication by $\sgn r$, then formally the indefinite and definite differential expressions $\tau$ and $\ell$ are related via $\tau=J\ell$, and hence $JA$ is the maximal operator associated with
$\tau$ in $L^2_{\vert r\vert}(\dR)$.
Observe that the indefinite Sturm-Liouville
operator $JA$ is neither symmetric nor selfadjoint in the Hilbert space $L^2_{\vert r\vert}(\dR)$ but $JA$ is
selfadjoint with respect to an indefinite inner product (which has $J$
as its Gramian); we shall say that $JA$ is a {\it $J$-selfadjoint} operator
in $L^2_{\vert r\vert}(\dR)$.

A modern topic in Sturm-Liouville theory is the study of qualitative and quantitative spectral properties
of indefinite Sturm-Liouville differential operators.
One of the standard approaches is to describe the spectrum $\sigma(JA)$
of the indefinite operator $JA$
via the selfadjoint
operator $A$ and its spectral properties. In the left-definite case,
i.e.,\  $\min\sigma(A)>0$, it follows
that the spectrum of $JA$ is real with a gap around $0$ and accumulates to $+\infty$ and $-\infty$,
see, e.g., \cite{CL89,KWZ04,Z05} and \cite{AI89,B74} for corresponding abstract results. If $A$ is semibounded
from below and the essential spectrum satisfies $\min\sigma_\ess(A)>0$, then the nonreal spectrum of
$JA$ consists of at most finitely many eigenvalues, the essential spectrum $\sigma_\ess(JA)$ is real with a gap around $0$,
and $\sigma(JA)\cap\dR$ accumulates to  $+\infty$ and $-\infty$, see, e.g. \cite{CL89,L82} and \cite{KMWZ03}. The spectral analysis of $JA$ in the case
$\min\sigma_\ess(A)\leq 0$ is more difficult; we refer to \cite{B07,BP10,KT09} for more details
and to \cite{BKT09,BT07,CN95,DL77,KKM09,KM07} for related questions and further references.

The main objective of the present paper is to prove a local estimate on the number of eigenvalues of $JA$
in terms of the number of eigenvalues of $A$ in gaps of the essential spectrum in the left-definite case,
i.e., $\min\sigma(A)>0$.
In this situation it is not difficult to see that for $0\leq a<b$ we have
$$
(a,b) \cap \sigma_\ess(A)=\emptyset
\quad \mbox{if and only if} \quad
\bigl((-b,-a) \cup (a,b)\bigr) \cap \sigma_\ess(JA)=\emptyset.
$$
Our main result Theorem~\ref{mainthm1} reads as follows: If
 $(a,b)\cap\sigma_\ess(A)=\emptyset$, then the number of eigenvalues $n_A(a,b)$ of $A$
 in $(a,b)$ differs at most by three from the number
$n_{JA}(-b,-a)+n_{JA}(a,b)$ of eigenvalues
of $JA$ in $(-b,-a)\cup(a,b)$,
$$
\bigl|n_A(a,b)-\bigl(n_{JA}(-b,-a)+n_{JA}(a,b)\bigr) \bigl|\,\leq 3.
$$
Under the assumption
that the coefficients $p,q$ and $r$ are symmetric with respect to $0$ the estimate on the number of
eigenvalues is improved in Theorem~\ref{mainthm2} for intervals $(a,b)$ with the property $0\leq a <\min\sigma(A)<b\leq\min\sigma_\ess(A)$.
The above estimates also yield results on accumulation
properties of eigenvalues of $JA$.
More precisely, if, e.g., $b\in\sigma_\ess(A)$ and
the eigenvalues of $A$ in $(a,b)$ accumulate to $b$, then the eigenvalues of $JA$ in the gaps
$(-b,-a)$ and $(a,b)$ of the essential spectrum accumulate to $-b$ or $b$.
This allows to transfer results on the accumulation (or non-accumulation)
of eigenvalues to the boundary of the essential spectrum of definite Sturm-Liouville operators
(as, e.g., the classical Kneser criterion from \cite{K11} or recent extensions of it in \cite{GST96,GUe,KT,KT9})
into the left-definite setting; see Section~\ref{Kneserli} for more details.

The paper is organized as follows. In Section 2 the operators $A$, $JA$, and the Dirichlet operators
associated with the restrictions $\ell_+$ and $\ell_-$ of the definite differential expression $\ell$
onto $(c,\infty)$ and $(-\infty,c)$ are introduced and some simple
properties of their spectra and
essential spectra are collected. Section 3 establishes the connection of the poles and zeros of the Titchmarsh-Weyl coefficients
$m_+$, $m_-$ associated with $\ell_+$ and $\ell_-$, respectively,
 with the poles and zeros of the Titchmarsh-Weyl coefficient $M$
 associated with $\tau$. This connection is then used to
describe the isolated eigenvalues of $A$ and $JA$ in terms of the poles and zeros of the
functions $m_+$, $m_-$ and $M$.
The representation of the function $M$ in terms of a Nevanlinna function in Proposition~\ref{mnevanprop} and the corresponding monotonicity properties
are the crucial ingredients in the proofs of our main results Theorem~\ref{mainthm1} and Theorem~\ref{mainthm2} in Section~4.
In Section 5  Kneser's criterion is applied in the left-definite setting and the general results are
illustrated in this situation.
Furthermore, a class of periodic
problems is considered (see also \cite{K10,MZ05} and \cite[$\S\,$12.8]{Z05} for a slightly different indefinite
periodic situation), and
a simple solvable problem is briefly discussed.

\section{Preliminaries on definite and indefinite Sturm-Liouville operators}
Let  $r,p^{-1},q\in L^1_\loc(\dR)$ be real valued functions with $p>0$ and $r\not=0$
almost everywhere.
We consider the differential expressions
\begin{equation*}
\tau =\frac{1}{r}\left(-\frac{d}{dx}\,p\,\frac{d}{dx}+q\right)
\quad\text{and}\quad
\ell =\frac{1}{\vert r\vert}\left(-\frac{d}{dx}\,p\,\frac{d}{dx}+q\right)
\end{equation*}
from \eqref{tau2222} and \eqref{ell2222}. In this section we collect some simple properties
on the spectra of the associated maximal operators.
It is assumed
that the following condition (I) holds for the weight function $r$:
\begin{itemize}
 \item [{\rm (I)}] There exists $c\in\dR$ such that the restriction
 $r_+:=r\upharpoonright_{(c,\infty)}$ is positive  almost everywhere and
the restriction $r_-:=r\upharpoonright_{(-\infty,c)}$ is negative almost everywhere.
\end{itemize}
The restrictions of the functions $p$ and $q$ onto the intervals
$(c,\infty)$ and $(-\infty,c)$ will be denoted by $p_+$, $q_+$, $p_-$ and $q_-$, respectively.

The space of all (equivalence classes of) complex valued measurable functions $f$ such that
$\vert f\vert^2 \vert r\vert\in L^1(\dR)$ is denoted by $L^2_{\vert r\vert}(\dR)$.
Equipped with the scalar product
\begin{equation}\label{defprod}
(f,g):=\int_{\dR} f(x)\,\overline{g(x)}\,\vert r(x)\vert\,dx,\qquad f,g\in L^2_{\vert r\vert}(\dR),
\end{equation}
this space is a Hilbert space. The maximal operator $Af=\ell f$ associated with the definite
Sturm-Liouville expression $\ell$
in $L^2_{\vert r\vert}(\dR)$ is defined on the dense subspace
\begin{equation*}
\sD:=\bigl\{f\in L^2_{\vert r\vert}(\dR):f,pf^\prime\,\,\text{locally absolutely continuous},
\,\,\ell f\in L^2_{\vert r\vert}(\dR) \bigr\}.
\end{equation*}
We denote by $\sD_+$ and $\sD_-$ the space of functions on $(c,\infty)$ and $(-\infty,c)$ which
are restrictions of functions from $\sD$ onto $(c,\infty)$ and $(-\infty,c)$,
respectively.
Throughout this paper it will be assumed that $A$ satisfies the following condition~(II):
\begin{itemize}
 \item [{\rm (II)}] The maximal operator $Af=\ell f$ defined on $\dom A=\sD$ is selfadjoint in $L^2_{\vert r\vert}(\dR)$ and $\min\sigma(A)>0$ holds.
\end{itemize}
Recall that $A$ is selfadjoint if and only if the definite Sturm-Liouville expression $\ell$ is in the limit
point case at both singular endpoints $+\infty$ and $-\infty$.

Besides the definite inner product $(\cdot,\cdot)$ in \eqref{defprod} the space $L^2_{\vert r\vert}(\dR)$ will
also be equipped with the indefinite inner product $[\cdot,\cdot]$ defined by
\begin{equation*}\label{indefprod}
[f,g]:=\int_{\dR} f(x)\,\overline{g(x)}\,r(x)\,dx,\qquad f,g\in L^2_{\vert r\vert}(\dR).
\end{equation*}
The space $L^2_r(\dR)=(L^2_{\vert r\vert}(\dR),[\cdot,\cdot])$ is a Krein space, the inner products $(\cdot,\cdot)$
and $[\cdot,\cdot]$ are connected via the fundamental symmetry $(Jf)(x)=\sgn(r(x))f(x)$, $x\in\dR$, that is, the relations
\begin{equation*}
 (Jf,g)=[f,g]\qquad \text{and}\qquad [f,g]=(Jf,g),\qquad f,g\in L^2_{\vert r\vert}(\dR),
\end{equation*}
hold, see, e.g., \cite{AI89,B74}.
Note that formally we have $\tau=J\ell$.
The maximal operator associated with $\tau$ coincides with $JA$. This
operator is selfadjoint with respect to
the indefinite inner product $[\cdot,\cdot]$;
we shall say that $JA$ is $J$-selfadjoint in the Hilbert space  $L^2_{\vert r\vert}(\dR)$.
As a consequence of condition (II) and well-known properties of $J$-nonnegative operators (see, e.g., \cite{B74})
we obtain the next proposition.

\begin{proposition}\label{indefprop1}
Assume that conditions (I) and (II) hold.
Then the indefinite Sturm-Liouville operator
\begin{equation*}
JA f=\tau f= \frac{1}{r}\bigl(-(pf^\prime)^\prime +qf \bigr),
\qquad f\in \dom JA=\sD,
\end{equation*}
is a $J$-selfadjoint operator  in $L^2_{\vert r\vert}(\dR)$ with
$$
\sigma(JA)\subset \mathbb R \quad \mbox{and} \quad 0 \in \rho(JA).
$$
Each eigenvalue $\lambda$ of $JA$ is simple, i.e.,\ $\dim \ker (JA-\lambda)=1$
and there is no Jordan chain of length greater than one.
\end{proposition}

For a more detailed analysis of the spectrum of $JA$ it is useful to consider the (definite) differential
expressions
\begin{equation}\label{ellpm}
\ell_+ =\frac{1}{r_+}\left(-\frac{d}{dx}\,p_+\,\frac{d}{dx}+q_+\right)\quad\text{and}\quad
\ell_- =-\frac{1}{r_-}\left(-\frac{d}{dx}\,p_-\,\frac{d}{dx}+q_-\right)
\end{equation}
and the associated differential operators in the subspaces
$L^2_{r_+}(c,\infty)$ and $L^2_{-r_-}(-\infty,c)$  which consist of restrictions of functions from
$L^2_{\vert r\vert}(\dR)$ onto the intervals $(c,\infty)$ and $(-\infty,c)$, respectively.
It follows from condition (I) that $L^2_{r_+}(c,\infty)$ and $L^2_{-r_-}(-\infty,c)$ equipped with
\begin{equation*}
\begin{split}
(h_1,h_2)_+&=\int_c^\infty h_1(x)\,\overline{h_2(x)}\,r_+(x)\,dx,\qquad h_1,h_2\in L^2_{r_+}(c,\infty),\\
(k_1,k_2)_-&=\int_{-\infty}^c k_1(x)\,\overline{k_2(x)}\,(-r_-(x))\,dx,\qquad k_1,k_2\in L^2_{-r_-}(-\infty,c),
\end{split}
\end{equation*}
are Hilbert spaces.
Since $\ell$ is in the limit point case at $+\infty$ and $-\infty$ it follows that the (restricted) differential
expressions $\ell_+$ and $\ell_-$ are in the limit point case at $+\infty$ and $-\infty$, respectively, and regular
at $c$.
In Section \ref{Poerlitz} below we will make use of the Lagrange identities
\begin{equation}\label{lagpm}
\begin{split}
(\ell_+ h_1,h_2)_+-(h_1,\ell_+ h_2)_+&=(p_+h_1^\prime)(c)\overline{h_2(c)}-h_1(c)\overline{(p_+h_2^\prime)(c)},\\
(\ell_- k_1,k_2)_--(k_1,\ell_- k_2)_-&=-(p_-k_1^\prime)(c)\overline{k_2(c)}+k_1(c)\overline{(p_-k_2^\prime)(c)},
\end{split}
\end{equation}
which hold for all $h_1,h_2\in\sD_+$ and $k_1,k_2\in\sD_-$. The
Dirichlet operators
\begin{equation}\label{Geraberg}
\begin{split}
B_+ h=\ell_+ h,&\quad\dom B_+=\bigl\{h\in\sD_+:h(c)=0\bigr\},\\
B_- k=\ell_- k,&\quad\dom B_-=\bigl\{k\in\sD_-:k(c)=0\bigr\},
\end{split}
\end{equation}
associated with $\ell_+$ and $\ell_-$ in \eqref{ellpm} are selfadjoint in the Hilbert spaces $L^2_{r_+}(c,\infty)$ and $L^2_{-r_-}(-\infty,c)$, respectively. Then the orthogonal sums $B=B_+\oplus B_-$ and
$JB=B_+\oplus (-B_-)$ are selfadjoint operators
in $L^2_{\vert r\vert}(\dR)$.
The next lemma on the spectrum and essential spectrum of the selfadjoint operators $A$, $B$ and $B_\pm$ will be useful later. For a closed
operator $T$ in a Hilbert space the {\it essential spectrum} $\sigma_\ess(T)$ consists of all $\lambda\in\dC$
such that $T-\lambda$ is not a Fredholm operator. Note that for a selfadjoint operator or a $J$-nonnegative operator
$T$ with $\rho(T)\not=\emptyset$ the set $\sigma_\ess(T)$ coincides with those spectral points which are no isolated eigenvalues of finite multiplicity.

\begin{lemma}\label{lemma1}
Assume that conditions (I) and (II) are satisfied.
For the spectra of the operators $A$, $B$ and $B_\pm$ the following relations hold:
\begin{itemize}
 \item [{\rm (i)}]  $\min\sigma(A)\leq\min\sigma(B)$ and
 $\min\sigma(A)\leq\min\sigma(B_\pm)$;
 \item [{\rm (ii)}] $\sigma_\ess(A)=\sigma_\ess(B_+)\cup\sigma_\ess(B_-)=\sigma_\ess(B)$ and
 $\sigma_\ess(B_\pm)\subset\sigma_\ess(A)$;
 \item [{\rm (iii)}]
 $\min\sigma_\ess(A)=\min\{\min\sigma_\ess(B_+),\min\sigma_\ess(B_-)\}
     =\min\sigma_\ess(B)$.
\item [{\rm (iv)}] Denote by $E_A$ and $E_B$ the spectral functions of
$A$ and $B$, respectively. For an open interval $\Delta$ with
$\Delta \cap \sigma_\ess(A) =\emptyset$ the estimate
$$
\vert\dim\ran E_A(\Delta) - \dim\ran E_B(\Delta)\,\vert\leq 1
$$
holds
if the corresponding quantities are finite. Otherwise
$\dim\ran E_A(\Delta)=\infty$ if and only if
$\dim\ran E_B(\Delta)=\infty$.
\end{itemize}
\end{lemma}
Observe that the case  $\dim\ran E_A(\Delta) =
\dim\ran E_B(\Delta)=\infty$ can only occur
 if one or both of the endpoints of $\Delta$
belong to the essential spectrum of $A$.

\begin{proof}
(i) Define the closed symmetric operators $S_+$ and $S_-$  in the Hilbert spaces $L^2_{r_+}(c,\infty)$ and $L^2_{-r_-}(-\infty,c)$
by
\begin{equation*}
S_+ h=\ell_+ h,\quad\dom S_+=\bigl\{h\in\sD_+:h(c)=(p_+h^\prime)(c)=0\bigr\}
\end{equation*}
and
\begin{equation*}
S_- k=\ell_- k,\quad\dom S_-=\bigl\{k\in\sD_-:k(c)=(p_-k^\prime)(c)=0\bigr\}.
\end{equation*}
As the orthogonal sum $S_+\oplus S_-$ is a restriction of $A$ it follows that $S_+\oplus S_-$ is a
symmetric operator with a lower bound larger or equal to $\min\sigma(A)$ which is positive by condition (II). Clearly, also
$S_+$ and $S_-$ are symmetric operators with lower bounds larger or equal to  $\min\sigma(A)$.
As $B_+$ and $B_-$ are the Friedrichs extensions of $S_+$ and $S_-$ (see \cite[Theorem~3 and Corollary~2]{R85})
also their lower bounds are larger or equal to $\min\sigma(A)$.  This shows the second statement in (i); the first assertion
in (i) is an immediate consequence.

The assertions in (ii) and (iii) follow from
\begin{equation}\label{resdiff}
\dim\ran \bigl((B-\lambda)^{-1}-(A-\lambda)^{-1}\bigr)=1,\qquad\lambda\in\rho(A)\cap\rho(B),
\end{equation}
whereas \eqref{resdiff} itself is a consequence of the fact that
$A$ and $B$ are selfadjoint extensions
of the symmetric operator $Rf=\ell f$, $\dom R=\{f\in\sD:f(c)=0\}$,
which has defect numbers $(1,1)$. This together with \cite[$\S\,$9.3, Theorem 3]{BS}
implies (iv).
\end{proof}

The following proposition on the essential spectrum of the indefinite Sturm-Liouville
operator $JA$ complements the statements in Proposition~\ref{indefprop1}.
It is a simple consequence of Lemma~\ref{lemma1} and
$\dim\ran((JA-\lambda)^{-1} - (JB-\lambda)^{-1}) =1$ for all
$\lambda\in\rho(JA) \cap \rho(JB)$. Note that $\rho(JA) \cap \rho(JB)\not=\emptyset$
by Proposition~\ref{indefprop1}.

\begin{proposition}\label{indefprop2}
Assume that conditions (I) and (II) hold.
Then the essential spectrum of the indefinite Sturm-Liouville operator $JA$ is given by
\begin{equation*}
\sigma_\ess(JA)=\sigma_\ess(JB)=
\bigl(\sigma_\ess(B_+)\cup\sigma_\ess(-B_-)\bigr)\,\subset\bigl(\sigma_\ess(A)\cup\sigma_\ess(-A)\bigr).
\end{equation*}
\end{proposition}

\section{The function $M$}\label{Poerlitz}

In this section we define a function $M$ with the help of Titchmarsh-Weyl coefficients $m_+$ and $m_-$
associated with the differential expressions $\ell_+$ and $\ell_-$ in \eqref{ellpm}. Since it turns out that
the zeros of $M$ coincide with the isolated eigenvalues of the indefinite Sturm-Liouville operator $JA$
we shall study the monotonicity properties of $M$, which then lead to eigenvalue estimates in the next section.
As a byproduct we also obtain a result on the size of the spectral gap of $JA$ around zero in Proposition~\ref{gapcor} below.

Assume throughout this section that conditions (I) and (II) hold and
let $B_+$ and $B_-$ be the selfadjoint Dirichlet operators in the Hilbert spaces
$L^2_{r_+}(c,\infty)$ and $L^2_{-r_-}(-\infty,c)$ from \eqref{Geraberg}, and let
$\lambda\in\rho(B_+)$ and $\mu\in\rho(B_-)$.
As $\ell_+$ and $\ell_-$ are in the limit point case at $+\infty$ and
at $-\infty$, respectively, there are
 unique (up to a constant multiple) solutions
  $h_\lambda\in\sD_+$ and $k_\mu\in \sD_-$
 of the differential equations
\begin{equation*}
\ell_+ h = \lambda h\qquad\text{and}\qquad \ell_- k = \mu k.
\end{equation*}
The functions $m_\pm:\rho(B_\pm)\rightarrow\dC$ are defined by
\begin{equation*}
 m_+(\lambda):=\frac{(p_+h_\lambda^\prime)(c)}{h_\lambda(c)}\qquad\text{and}\qquad
 m_-(\mu):=\frac{(p_-k_\mu^\prime)(c)}{k_\mu(c)}.
\end{equation*}

It is obvious that the poles of $m_\pm$ coincide with the isolated eigenvalues of $B_\pm$,
and that the poles of the function $\lambda\mapsto m_-(-\lambda)$ coincide with the isolated eigenvalues of
$-B_-$.
The functions $m_\pm$ are holomorphic on $\rho(B_\pm)$, they do not admit analytic extensions
to points of $\sigma(B_\pm)$, and they are
symmetric with respect to the real axis, i.e.,
\begin{equation*}
m_+(\bar\lambda)=\overline{m_+(\lambda)}\quad\text{and}\quad m_-(\bar\mu)=\overline{m_-(\mu)}.
\end{equation*}
If
we fix solutions $h_\lambda$ and $k_\mu$ with $h_\lambda(c)=1$ and $k_\mu(c)=1$ it follows from
\eqref{lagpm} that the relations
\begin{equation*}
\begin{split}
(\lambda-\bar\lambda)(h_\lambda,h_\lambda)_+&=(\ell_+ h_\lambda,h_\lambda)_+-(h_\lambda,\ell_+ h_\lambda)_+
=m_+(\lambda)-\overline{m_+(\lambda)},\\
(\bar\mu-\mu)(k_\mu,k_\mu)_-&=(k_\mu,\ell_- k_\mu)_--(\ell_- k_\mu,k_\mu)_- = m_-(\mu)-\overline{m_-(\mu)},
\end{split}
\end{equation*}
hold. Therefore, $\pm m_\pm$ are so-called Nevanlinna functions.
Recall that a complex-valued function $N$ is said to be a  Nevanlinna function if $N$ is holomorphic on $\dC\backslash\dR$ and the properties
\begin{equation*}
N(\bar\lambda)=\overline{N(\lambda)}\quad\text{and}\quad \frac{\Im N(\lambda)}{\Im\lambda} \geq 0
\end{equation*}
hold for all $\lambda\in\dC\backslash\dR$. For later purposes it is important to note
that a Nevanlinna function $N$ is monotone increasing on real intervals which belong to its domain of holomorphy and that $N$ is equal to a
constant on such an interval if and only if $N$ is a constant
function on $\mathbb C$, see, e.g., \cite{KK}.

In the following we will relate the zeros and poles of the function
\begin{equation}\label{mfct}
 M(\lambda):=m_+(\lambda)-m_-(-\lambda),\qquad\lambda\in\rho(B_+)\cap\rho(-B_-),
\end{equation}
with the eigenvalues of the operators $JA$ and $JB$.
Clearly, the domain of holomorphy of $M$ contains the interval $(-\min\sigma(B_-),\min\sigma(B_+))$,
the poles of $M$ in $[\min\sigma(B_+),\infty)$ and $(-\infty,-\min\sigma(B_-)]$ coincide with the poles of $\lambda \mapsto m_+(\lambda)$ and
$\lambda \mapsto m_-(-\lambda)$, respectively.
Hence each pole of $M$ in $[\min\sigma(B_+),\infty)$ is an isolated eigenvalue of
$B_+$ and each pole of $M$ in $(-\infty,-\min\sigma(B_-)]$ is an isolated eigenvalue of $-B_-$.
Therefore, each pole of $M$ is an isolated eigenvalue of the operator $JB=B_+\oplus (-B_-)$, and,
vice versa, every isolated eigenvalue of $JB$ is a pole of $M$.
 This shows assertion (ii) in the next proposition.

\begin{proposition}\label{polezeroprop}
For $\lambda\not\in\sigma_\ess(JA)$ the following assertions hold:
\begin{itemize}
 \item [{\rm (i)}] $\lambda\in\sigma_p(JA)$ if and only if $\lambda$ is a zero of $M$;
 \item [{\rm (ii)}] $\lambda\in\sigma_p(JB)$ if and only if $\lambda$ is a pole of $M$.
\end{itemize}
\end{proposition}

\begin{proof}
It remains to show assertion (i). For this observe first that  $\lambda\not\in\sigma_\ess(JA)=\sigma_\ess(B_+)\cup\sigma_\ess(-B_-)$ is an eigenvalue of $JA$
with corresponding eigenfunction $f_\lambda\in\sD$ if and only if $f_\lambda=h_\lambda\oplus k_{-\lambda}$,
where $h_\lambda\in\sD_+$ and $k_{-\lambda}\in\sD_-$ are
the (nontrivial) restrictions of
$f_\lambda$ onto $(c,\infty)$ and
$(-\infty,c)$, respectively, which satisfy
the differential equations
\begin{equation}\label{diff1}
\ell_+ h_\lambda=\lambda h_\lambda,\qquad \ell_-k_{-\lambda}=-\lambda k_{-\lambda},
\end{equation}
and the conditions
\begin{equation}\label{diff2}
h_\lambda(c)=k_{-\lambda}(c),\qquad (p_+h_\lambda^\prime)(c)=(p_-k_{-\lambda}^\prime)(c).
\end{equation}
As a simple consequence we conclude
\begin{equation*}
\sigma_p(JA)\cap\sigma_p(B_+)\cap\rho(-B_-) =\emptyset
\quad\text{and}\quad\sigma_p(JA)\cap\sigma_p(-B_-)\cap\rho(B_+) = \emptyset.
\end{equation*}
Furthermore, $\sigma_p(B_+)\cap\sigma_p(-B_-)=\emptyset$ by Lemma~\ref{lemma1}~(i)
and condition (II) and, hence,
it is sufficient to prove the equivalence in (i) for
$\lambda\in\rho(B_+)\cap\rho(-B_-)$.

Assume first that $\lambda\in\sigma_p(JA)\cap\rho(B_+)\cap\rho(-B_-)$, so that \eqref{diff1} and \eqref{diff2} hold for some corresponding
eigenfunction
 $f_\lambda=h_\lambda\oplus k_{-\lambda}$ of $JA$ and $h_\lambda(c)=k_{-\lambda}(c)\not =0$.
This yields
\begin{equation}\label{mpmm}
m_+(\lambda)=\frac{(p_+h_\lambda^\prime)(c)}{h_\lambda(c)}=\frac{(p_-k_{-\lambda}^\prime)(c)}{k_{-\lambda}(c)}
=m_-(-\lambda)
\end{equation}
and hence $M(\lambda)=0$. Conversely, let $\lambda\in\rho(B_+)\cap\rho(-B_-)$ be a zero of $M$ and let $h_\lambda\in\sD_+$ and $k_{-\lambda}\in\sD_-$ be (nontrivial) solutions of \eqref{diff1}
which satisfy $h_\lambda(c)=k_{-\lambda}(c)\not= 0$. From $M(\lambda)=0$ we obtain $m_+(\lambda)=m_-(-\lambda)$ and it follows from
\eqref{mpmm} that also the second condition in \eqref{diff2} is satisfied by $h_\lambda$ and $k_{-\lambda}$. Therefore $f_\lambda:=h_\lambda\oplus
k_{-\lambda}$ belongs to $\sD$ and is an eigenfunction of $JA$ corresponding to $\lambda$.
\end{proof}

In a similar way as in Proposition~\ref{polezeroprop}  the eigenvalues
of $A$ and of $B=B_+\oplus B_-$ are related to the poles and zeros of the
functions $m_+$ and $m_-$. Since the isolated eigenvalues
of $B_+$ and $B_-$ coincide with the poles of $m_+$ and $m_-$  it is clear that $\lambda$ is an  eigenvalue
of $B$ if and only if $\lambda$ is a pole of $m_+$ or $m_-$; this shows item (ii) in the
next proposition.
For the convenience of the reader also the first item will be shown in detail.

\begin{proposition}\label{polezeroprop2}
For $\lambda\not\in\sigma_\ess(A)$ the following assertions hold:
\begin{itemize}
 \item [{\rm (i)}] $\lambda\in\sigma_p(A)$ if and only if $\lambda$
is a either a zero of $m_+-m_-$ or a pole of both $m_+$ and $m_-$;
 \item [{\rm (ii)}] $\lambda\in\sigma_p(B)$ if and only if $\lambda$
is a pole of $m_+$ or of $m_-$.
\end{itemize}
\end{proposition}

\begin{proof}
It remains to show assertion (i). For this observe first that  $\lambda\not\in\sigma_\ess(A)=\sigma_\ess(B_+)\cup\sigma_\ess(B_-)$ is an eigenvalue of $A$
with corresponding eigenfunction $f_\lambda\in\sD$ if and only if $f_\lambda=h_\lambda\oplus k_{\lambda}$,
where $h_\lambda\in\sD_+$ and $k_{\lambda}\in\sD_-$ are the (nontrivial) restrictions
of $f_\lambda$ onto $(c,\infty)$ and
$(-\infty,c)$, respectively, which
satisfy the differential equations
\begin{equation}\label{diff1A}
\ell_+ h_\lambda=\lambda h_\lambda,\qquad \ell_-k_{\lambda}=\lambda k_{\lambda},
\end{equation}
and the conditions
\begin{equation}\label{diff2A}
h_\lambda(c)=k_{\lambda}(c),\qquad (p_+h_\lambda^\prime)(c)=(p_-k_{\lambda}^\prime)(c).
\end{equation}
Hence
\begin{equation*}
\sigma_p(A)\cap\sigma_p(B_+)\cap\rho(B_-) =\emptyset
\quad\text{and}\quad\sigma_p(A)\cap\sigma_p(B_-)\cap\rho(B_+) = \emptyset .
\end{equation*}

Assume first that $\lambda$ is an eigenvalue of $A$. Then either
$\lambda\in\rho(B_+) \cap \rho(B_-)$
or $\lambda\in\sigma_p(B_+)\cap\sigma_p(B_-)$. In the first case \eqref{diff2A} implies $m_+(\lambda)=m_-(\lambda)$ and hence
$\lambda$ is a zero of $m_+-m_-$. In the second case $\lambda$ is a pole of both
$m_+$ and $m_-$.
Conversely, let
$h_\lambda\in\sD_+$ and $k_{\lambda}\in\sD_-$ be nontrivial
solutions of \eqref{diff1A}
which satisfy $h_\lambda(c)=k_{\lambda}(c)$. If $\lambda$ is a zero of $m_+-m_-$ then
$\lambda\in\rho(B_+)\cap\rho(B_-)$ and $h_\lambda(c)=k_{\lambda}(c)\not=0$, so that
the assumption $m_+(\lambda)-m_-(\lambda)=0$ implies the second condition in \eqref{diff2A}. Therefore $f_\lambda:=h_\lambda\oplus
k_{\lambda}$ belongs to $\sD$ and is an eigenfunction of $A$
corresponding to $\lambda$. If $\lambda$ is a pole
of $m_+$ and of $m_-$, then $\lambda \in\sigma_p(B_+)\cap\sigma_p(B_-)$
and
hence the nontrivial solutions $h_\lambda\in\sD_+$ and $k_{\lambda}\in\sD_-$ of
\eqref{diff1A} satisfy $h_\lambda(c)=k_{\lambda}(c)=0$ and
$(p_+h_{\lambda}^\prime)(c)\not= 0$ and $(p_-k_\lambda^\prime)(c)\not= 0$.
Since $h_\lambda$ and $k_{\lambda}$ are unique up to a constant
multiple, it follows that the function
$$
f_\lambda:=(\nu h_\lambda)\oplus k_{\lambda},\qquad\text{where}\,\,\, \nu:=\frac{(p_-k_\lambda^\prime)(c)}{(p_+h_{\lambda}^\prime)(c)},
$$
belongs to  $\sD$ and is an eigenfunction of $A$
corresponding to $\lambda$.
\end{proof}

As a consequence of the above propositions we obtain a statement on the size of the spectral
gap of $JA$ around $0$; cf. Proposition~\ref{indefprop1}. We mention that item (ii) in the next proposition
can also be deduced from \cite[Behauptung 3]{T74} applied to the inverses of $A$ and $JA$.

\begin{proposition}\label{gapcor}
Assume that conditions (I) and (II) are satisfied. Then the following statements hold:
\begin{itemize}
\item [{\rm (i)}]
If $\min\sigma(A)<\min\sigma_\ess(A)$, then $[-\min\sigma(A),\min\sigma(A)]\subset\rho(JA)$;
\item [{\rm (ii)}] If $\min \sigma(A)=
\min\sigma_\ess (A)$, then $(-\min\sigma(A),\min\sigma(A))\subset\rho(JA)$.
\end{itemize}
\end{proposition}

\begin{proof}
Let $0< \lambda_1:=\min\sigma(A)$. We show first that the inclusion
\begin{equation}\label{biggap2}
\bigl(-\min\sigma(A),\min\sigma(A)\bigr)\subset\rho(JA)
\end{equation}
holds under
any of the assumptions in (i) and (ii), i.e.,
$\min\sigma(A)\leq\min\sigma_\ess(A)$.
In fact, by Lemma \ref{lemma1} and Proposition \ref{polezeroprop2}~(i) we have that $m_+ - m_-$ is holomorphic and does not vanish on $(-\lambda_1,\lambda_1)$. Since $m_+$ and $-m_-$ are Nevanlinna functions, it follows that $m_+$ is increasing and $m_-$ is decreasing on $(-\lambda_1,\lambda_1)$. Thus, the images of $m_+$ and $m_-$ of $(-\lambda_1,\lambda_1)$ are intervals which do not intersect. Consequently, the images of $m_+$ and $m_-(-\cdot)$ of $(-\lambda_1,\lambda_1)$ are also intervals which do not intersect, so that $M$ does not vanish on $(-\lambda_1,\lambda_1)$. This, together with Proposition \ref{polezeroprop}~(i) implies that there are no eigenvalues of $JA$ in $(-\lambda_1, \lambda_1)$, which yields \eqref{biggap2} and hence assertion (ii)
has been shown

In order to prove assertion (i) it remains to verify that $\lambda_1$ and $-\lambda_1$ are not
eigenvalues of
$JA$ if  $\min\sigma(A)<\min\sigma_\ess(A)$ holds. We provide the argument for $\lambda_1$; a similar reasoning
applies to $-\lambda_1$.
 By Proposition~\ref{polezeroprop2}~(i) either $m_+(\lambda_1)=
m_-(\lambda_1)$ or both functions $m_+$ and $m_-$ have a pole at $\lambda_1$.
In the first case we have
$$
M(\lambda_1)=m_+(\lambda_1)-m_-(-\lambda_1)<m_+(\lambda_1)-m_-(\lambda_1)=0
$$
since $-m_-$ is a nonconstant Nevanlinna function which is holomorphic on $(-\infty,\lambda_1]$ ($m_-$ is not constant as otherwise $\sigma(B_-)=\emptyset$). In particular, $M(\lambda_1)\not=0$ and hence $\lambda_1$ is not
an eigenvalue of $JA$ by  Proposition~\ref{polezeroprop}~(i). If $m_+$ and $m_-$ both have a pole
at $\lambda_1$ then it follows from the holomorphy of $m_-$ on $(-\infty,\lambda_1)$
that
the function $\lambda \mapsto m_-(-\lambda)$ is holomorphic in $\lambda_1$
and, hence,
$M(\cdot)=m_+(\cdot)-m_-(-\cdot)$ has a pole at $\lambda_1$. Again
Proposition~\ref{polezeroprop}~(i) implies $\lambda_1\in\rho(JA)$.
\end{proof}

Note that
under the assumptions in Proposition~\ref{gapcor} an upper estimate for the spectral gap of $JA$
can be given: If the smallest eigenvalue $\lambda_1=\min\sigma(A)$ of $A$ is a zero of $m_+-m_-$,
then it can be shown that the  largest negative eigenvalue
$\lambda_{1,-}(JA)$ and the
smallest positive eigenvalue $\lambda_{1,+}(JA)$ (i.e. the endpoints of the spectral gap) of $JA$ satisfy
$$\min\sigma(-B_-)< \lambda_{1,-}(JA) \qquad\text{and}\qquad \lambda_{1,+}(JA)<\min\sigma(B_+).$$
In the case that $\lambda_1$ is a pole of $m_+$ and $m_-$ the eigenvalues $\lambda_{1,-}(JA)$ and
$\lambda_{1,+}(JA)$ the above estimates hold with $\min\sigma(-B_-)$ and $\min\sigma(B_+)$ replaced by
the second largest eigenvalue of $-B_-$ and the second smallest eigenvalue of $B_+$ if these
eigenvalues exist, and by $\min\sigma_\ess(B_+)$ and $\max\sigma_\ess(-B_-)$ otherwise.

The next proposition and corollary will play an important role in the proof of our main result in the next section.

\begin{proposition}\label{mnevanprop}
The function $M$ admits the representation
\begin{equation}\label{mrep}
M(\lambda)=\frac{-1}{\alpha+\lambda N(\lambda)},
\end{equation}
where $N$ is a  Nevanlinna function which is not identically zero on $\mathbb C$ and $\alpha$ is a real constant. In particular,
$M$ is monotonously increasing (monotonously decreasing)
on subintervals of $\dR^+$ $($$\dR^-$, respectively$)$ which belong to its domain of holomorphy.
\end{proposition}

\begin{proof}
Let $\lambda\in\rho(JA)\cap\rho(B_+)\cap\rho(-B_-)$ and let  $h_\lambda,h_0\in\sD_+$, $k_{-\lambda},k_0\in\sD_-$ be the unique functions that satisfy
\begin{equation}\label{Erfurt1}
\ell_+ h_\lambda=\lambda h_\lambda,\,\,\,\ell_+h_0=0,\,\,\,\ell_-k_{-\lambda}=-\lambda k_{-\lambda},\,\,\,\ell_-k_0=0.
\end{equation}
and the conditions
\begin{equation}\label{Erfurt2}
\begin{split}
h_\lambda(c)&=k_{-\lambda}(c),\qquad (p_- k_{-\lambda}^\prime)(c)-(p_+h_\lambda^\prime)(c)=1,\\
h_0(c)&=k_0(c),\qquad \,\,\,\,\,\,\,\,\,(p_- k_0^\prime)(c)-(p_+h_0^\prime)(c)=1.
\end{split}
\end{equation}
We claim that the functions $f_\lambda=h_\lambda\oplus k_{-\lambda}$ and $f_0=h_0\oplus k_0$
are related via
\begin{equation}\label{Erfurt3}
f_\lambda = f_0 + \lambda (JA-\lambda)^{-1}f_0.
\end{equation}
For this observe that $ \lambda (JA-\lambda)^{-1}f_0\in \sD$ and, hence,
 $g:= f_0 +\lambda (JA-\lambda)^{-1}f_0$ satisfies the same conditions as $f_0=h_0\oplus k_0$ in \eqref{Erfurt2}.
Hence, if we write $g$ in the form
$g=h\oplus k$ with $h\in\sD_+$ and $k\in\sD_-$, then we have
\begin{equation*}
h(c)=k(c)\qquad\text{and}\qquad (p_- k^\prime)(c)-(p_+h^\prime)(c)=1.
\end{equation*}
As
$(\tau-\lambda) \lambda (JA-\lambda)^{-1}f_0=\lambda f_0$ we conclude
that $\pm \ell_\pm - \lambda$ applied to the restriction of  $\lambda (JA-\lambda)^{-1}f_0$
onto $(c,\infty)$ and $(-\infty,c)$ equals $\lambda h_0$ and $\lambda k_0$, respectively.
Therefore
\begin{equation*}
\begin{split}
(\ell_+ -\lambda) h&=(\ell_+ -\lambda) h_0 +\lambda h_0 = 0,\\
(\ell_- +\lambda) k&=(\ell_- +\lambda) k_0 -\lambda k_0 = 0,
\end{split}
\end{equation*}
and it follows that $h$ and $k$ satisfy the equations $\ell_+ h=\lambda h$ and $\ell_- k=-\lambda k$. Since the function $f_\lambda=h_\lambda\oplus k_{-\lambda}$ in \eqref{Erfurt1} and \eqref{Erfurt2}
is unique we obtain \eqref{Erfurt3}.

From
\begin{equation*}
\begin{split}
M(\lambda)&=\frac{(p_+h_\lambda^\prime)(c)-(p_-k_{-\lambda}^\prime)(c)}{f_\lambda(c)}=-\frac{1}{f_\lambda(c)},\\
M(0)&=\frac{(p_+h_0^\prime)(c)-(p_-k_0^\prime)(c)}{f_0(c)}=-\frac{1}{f_0(c)}\in\dR,
\end{split}
\end{equation*}
we conclude $M(\lambda) \ne 0$ for $\lambda\in\rho(JA)\cap\rho(B_+)\cap\rho(-B_-)$  and
 $M(0)\ne 0$. With \eqref{lagpm} we have
\begin{equation*}
\begin{split}
\lambda [f_\lambda,f_0]&=\lambda (h_\lambda,h_0)_+-\lambda (k_{-\lambda},k_0)_-\\
&=(\ell_+ h_\lambda,h_0)_+-(h_\lambda,\ell_+h_0)_++(\ell_-k_{-\lambda},k_0)_--(k_{-\lambda},\ell_-k_0)_-\\
&=(p_+h_\lambda^\prime)(c)\overline{h_0(c)}-h_\lambda(c)\overline{(p_+h_0^\prime)(c)}-(p_-k_{-\lambda}^\prime)(c)
\overline{k_0(c)}+k_{-\lambda}(c)\overline{(p_-k_0^\prime)(c)}\\
&=f_\lambda(c)-\overline{f_0(c)}.
\end{split}
\end{equation*}
Thus $-M^{-1}$ admits the representation
\begin{equation*}
 -M^{-1}(\lambda)=-M^{-1}(0)+\lambda[f_\lambda,f_0]
\end{equation*}
and with \eqref{Erfurt3} and $N(\lambda):=[(1+\lambda(JA-\lambda)^{-1}) f_0,f_0]$ we obtain
\begin{equation}\label{reprepm}
-M^{-1}(\lambda)=-M^{-1}(0)+\lambda N(\lambda).
\end{equation}
A simple calculation shows
\begin{equation*}
\Im N(\lambda)=\Im \lambda \bigl(A(JA-\lambda)^{-1}f_0,(JA-\lambda)^{-1}f_0\bigr)
\end{equation*}
and since $A$ is nonnegative by condition (II) it follows that $N$ is a Nevanlinna function, i.e.,
$M$ admits a representation of the from \eqref{mrep} with $\alpha := -M(0)^{-1}$.

Note  that $N$ is not equal to
zero on real intervals which belong to its domain of holomorphy, as otherwise $N\equiv 0$
and \eqref{reprepm} imply that $M$ in \eqref{mfct} is equal to a constant, so that the Titchmarsh-Weyl
coefficients
$\lambda\mapsto m_+(\lambda)$ and $\lambda\mapsto m_-(-\lambda)$ of $B_+$ and $-B_-$ differ by a real constant;
a contradiction to $\sigma(B_+)\cap\sigma(-B_-)=\emptyset$.
Now the remaining statements of
Theorem \ref{mnevanprop} follow from the fact that the Nevanlinna function $N$ is monotonously increasing on real intervals which belong to its domain of holomorphy.
\end{proof}

\begin{corollary}\label{corcor}
In between two consecutive positive (negative) poles $\nu,\nu^{\prime}$ of $M$  such that the interval $(\nu,\nu^{\prime})$ belongs to the domain of holomorphy of $M$
there is a unique zero of $M$. Similarly, in between two consecutive positive (negative) zeros $\eta,\eta^{\prime}$ of $M$ such that $M$ is meromorphic
in an open neighbourhood of
the interval $(\eta,\eta^{\prime})$ there is a unique pole of $M$ in
 $(\eta,\eta^{\prime})$.
\end{corollary}

The poles of $M$ in $[\min\sigma(B_+),\infty)$
(resp.\ $(-\infty,-\min\sigma(B_-))$)
coincide with the poles of $\lambda\mapsto m_+(\lambda)$ (resp. $\lambda\mapsto m_-(-\lambda)$),
and hence with the isolated eigenvalues of $B_+$
(resp.\ $-B_-$). From this we obtain
with Corollary \ref{corcor}  and Proposition~\ref{polezeroprop}~(i)
interlacing results of the positive eigenvalues of $JA$ with respect to the eigenvalues of $B_+$
and of the negative eigenvalues of $JA$ with respect to the eigenvalues of $-B_-$.

\begin{corollary}
In between any two consecutive isolated eigenvalues
of $B_+$ ($-B_-$) in a gap of $\sigma_\ess(B_+)$ ($\sigma_\ess(-B_-)$) there is exactly one isolated eigenvalue of $JA$.
Conversely, in between any two consecutive isolated positive (negative) eigenvalues
of $JA$ in a gap of $\sigma_\ess(JA)$ there is exactly one isolated eigenvalue of $B_+$ ($-B_-$, respectively).
\end{corollary}

\section{Eigenvalue estimates in gaps of the essential spectrum}

In this section we
prove estimates on the number of eigenvalues of $JA$ in a gap of the essential spectrum.
Recall that all eigenvalues of the operators $A$, $JA$, $B_+$, $-B_-$ and,
hence, $JB$ are simple.
For a selfadjoint or $J$-selfadjoint operator $T$ and a real interval $(a,b)$ such that $(a,b)\cap\sigma_\ess(T)=\emptyset$ the number of  eigenvalues
 of $T$ in $(a,b)$
will be denoted by $n_T(a,b)$, i.e.,
\begin{equation*}
n_T(a,b)=\sharp\bigl\{\lambda\in\sigma_p(T):\lambda\in (a,b)\bigr\}.
\end{equation*}

The following theorem is the main result of this note. It provides a local estimate on the
number of eigenvalues of $JA$ in terms of the number of eigenvalues of $A$ in a gap of the essential spectrum.
Recall that by Lemma \ref{lemma1}  and  Proposition~\ref{indefprop2} we have for $0\leq a<b$
$$
(a,b) \cap \sigma_\ess(A) = \emptyset \quad
\mbox{if and only if} \quad
\bigl( (-b,-a) \cup (a,b)\bigr) \cap \sigma_\ess(JA) = \emptyset .
$$

\begin{theorem}\label{mainthm1}
Assume that conditions (I) and (II) hold for the Sturm-Liouville operator $A$ and let
$JA$ be the corresponding indefinite Sturm-Liouville differential
operator. For $0\leq a<b$ such that $(a,b)\cap \sigma_\ess(A) = \emptyset$
the estimate
\begin{equation}\label{estimate}
\bigl|n_A(a,b)-\bigl(n_{JA}(-b,-a)+n_{JA}(a,b)\bigr) \bigl|\leq 3
\end{equation}
is valid if the corresponding quantities are finite; otherwise
$$
n_{A}(a,b)=\infty\quad\text{if and only if}\quad n_{JA}(-b,-a)+n_{JA}(a,b)=\infty.
$$
\end{theorem}

Observe that the case  $ n_{A}(a,b)=\infty $
(and, hence, $n_{JA}(-b,-a)+n_{JA}(a,b)=\infty$)
 can only occur if one or both of the endpoints $a$ and $b$
belong to the essential spectrum of $A$ which implies the following corollary.

\begin{corollary}\label{cor2345}
Let $A$, $JA$ and $(a,b)$ be as in Theorem \ref{mainthm1} and assume, in addition, that
$b\in\sigma_\ess(A)$, or, equivalently, that $b\in\sigma_\ess(JA)$ or  $-b\in\sigma_\ess(JA)$. Then the
eigenvalues of $A$ in $(a,b)$ accumulate to $b$ if and only if the eigenvalues of $JA$ in
$(-b,-a) \cup (a,b)$ accumulate to $b$ or $-b$.
\end{corollary}

\begin{proof}[Proof of Theorem~\ref{mainthm1}]
Let $(a,b)$ be as in the theorem and
suppose that the number $n_A(a,b)$ of eigenvalues of $A$ in $(a,b)$ is finite.
Since the eigenvalues of $A$ are all simple, $n_A(a,b)$ coincides with $\dim\ran E_A(a,b)$
and we conclude from Lemma~\ref{lemma1}~(iv) that $\dim\ran E_B(a,b)$ differs at most by one from $n_A(a,b)$.
Hence the number of eigenvalues $n_{B_+}(a,b)+n_{-B_-}(-b,-a)$ of $JB=B_+\oplus -B_-$ differs at most by one from $n_A(a,b)$ and by Proposition~\ref{polezeroprop}~(ii) the same holds true for the number of poles of the function $M$
in $(-b,-a)\cup(a,b)$.
It follows from Corollary~\ref{corcor} that $M$ has at least $n_A(a,b)-3$ zeros in $(-b,-a)\cup(a,b)$, so that
$$
n_{JA}(-b,-a)+n_{JA}(a,b)\geq n_A(a,b)-3
$$
by Proposition~\ref{polezeroprop}~(i). In order to show  \eqref{estimate}
suppose that
$$
n_{JA}(-b,-a)+n_{JA}(a,b)> n_A(a,b)+3.
$$
In this case Proposition~\ref{polezeroprop}~(i) yields that there are more than $n_A(a,b)+3$ zeros of $M$ in $(-b,-a)\cup(a,b)$ and hence there are
more than $n_A(a,b)+1$ poles of $M$ in $(-b,-a)\cup(a,b)$ by Corollary~\ref{corcor}. On the other hand, by
the above reasoning the number of poles of $M$ in $(-b,-a)\cup(a,b)$ differs at most by one from $n_A(a,b)$, a contradiction and \eqref{estimate} is shown.

From \eqref{estimate} it follows that for $a$ (or $b$)
in the essential spectrum of $A$ the quantity  $n_{A}(a,b)$ is finite
if and only if the quantity $n_{JA}(-b,-a)+n_{JA}(a,b)$ is finite.
\end{proof}

Let us now consider the case where the coefficients $p,q$ and $r$ satisfy some symmetry properties with respect to $c$. For simplicity we assume $c=0$ and for the following we suppose:
\begin{itemize}
\item [(III)] The functions $p$ and $q$ are even and $r$ is odd, i.e.
$$
p(x)=p(-x), \quad q(x)=q(-x) \quad \mbox{and}
\quad r(x)=-r(-x) \quad \mbox{for a.e. }x\in\dR.
$$
\end{itemize}
\vskip 0.3cm
Obviously, (III) implies for the operators $B_+$
and $B_-$ from \eqref{Geraberg}
\begin{equation*}\label{Geraberg2}
\sigma(B_+) =  \sigma(B_-) \quad \mbox{and} \quad
 \sigma_\ess(B_+) = \sigma_\ess(B_-).
\end{equation*}
Together with Proposition \ref{indefprop2} we conclude
$$
\sigma_\ess(JA) = \sigma_\ess(A) \cup  \sigma_\ess(-A).
$$
Furthermore, if $h_\lambda\in\sD_+$ and $k_\lambda\in\sD_-$ are related via $h_\lambda(x)=k_\lambda(-x)$, $x\in\dR^+$, then we have
\begin{equation*}
\ell_+ h_\lambda=\lambda h_\lambda,\qquad\text{if and only if}\qquad \ell_-k_\lambda=\lambda k_\lambda.
\end{equation*}
Together with $(p_+h_\lambda^\prime)(0)=-(p_-k_\lambda^\prime)(0)$ this implies $m_+(\lambda)=-m_-(\lambda)$
and it follows that the function $M$ in \eqref{mfct} is given by
\begin{equation}\label{msym}
M(\lambda)=m_+(\lambda)+m_+(-\lambda).
\end{equation}
Observe that by Proposition~\ref{polezeroprop}~(i) the eigenvalues of $JA$ are symmetric with respect to
zero. In particular $n_{JA}(a,b)=n_{JA}(-b,-a)$ in Theorem~\ref{mainthm1}. This implies the following statement
which is a slight improvement of the estimate \eqref{estimate} in Theorem~\ref{mainthm1} if condition (III) holds
and $n_A(a,b)$ is even.

\begin{corollary}\label{cor234}
Let the assumptions be as in Theorem~\ref{mainthm1} and assume, in addition, that condition (III) is satisfied.
If $n_A(a,b)$ is even, then the estimates
\begin{equation*}
\left|\tfrac{1}{2}n_A(a,b)-n_{JA}(a,b)\right|=\left|\tfrac{1}{2}n_A(a,b)-n_{JA}(-b,-a)\right|\leq 1
\end{equation*}
are valid.
\end{corollary}

The estimates in Theorem~\ref{mainthm1} and Corollary~\ref{cor234} will be further improved
in Theorem~\ref{mainthm2} below for the case that
condition (III) holds and instead of a gap in the essential spectrum we
consider the special
situation of an interval $(\alpha,\beta)$ with
 $0\leq \alpha < \min\sigma(A)< \beta \leq\min\sigma_\ess(A)$. It is worth mentioning that the following result for the comparison of the quantities
$n_{JA}(\alpha,\beta)$ and $n_{A}(\alpha,\beta)$ is optimal.


\begin{theorem}\label{mainthm2}
Assume that conditions (I), (II), and (III) hold for the Sturm-Liouville operator $A$,
that $\min\sigma(A)<\min\sigma_\ess(A)$
and let
$JA$ be the corresponding indefinite Sturm-Liouville differential
operator.
For $0\leq \alpha<\min\sigma(A)<\beta\leq\min\sigma_\ess(A)$ the following holds:
\begin{equation}\label{DaSchaust}
n_{JA}(\alpha,\beta)=n_{JA}(-\beta,-\alpha)=
\begin{cases}
\frac{1}{2}n_A(\alpha,\beta) & \text{if}\,\,n_A(\alpha,\beta)\,\,\text{is even},\\
\frac{1}{2}(n_A(\alpha,\beta)\pm 1) & \text{if}\,\,n_A(\alpha,\beta)\,\,\text{is odd},
\end{cases}
\end{equation}
where one of the quantities $n_{A}(\alpha,\beta)$, $n_{JA}(\alpha,\beta)$, $n_{JA}(-\beta,-\alpha)$ is infinite if and only if all the quantities $n_{A}(\alpha,\beta)$, $n_{JA}(\alpha,\beta)$, $n_{JA}(-\beta,-\alpha)$ are infinite.

In particular, the eigenvalues of $A$ below $\min \sigma_\ess(A)$
accumulate to $\min \sigma_\ess(A)$  if and only if the eigenvalues of $JA$ in
the interval
$(-\min \sigma_\ess(A), \min \sigma_\ess(A))$ accumulate to $-\min \sigma_\ess(A)$
and to $\min \sigma_\ess(A)$.
\end{theorem}
Observe that the case  $n_{A}(\alpha,\beta)=\infty $
(and, hence, $n_{JA}(\alpha,\beta)=n_{JA}(-\beta,-\alpha)=\infty$)
 can only occur if the endpoint $\beta$
belongs to the essential spectrum of $A$.

\begin{proof}
Let $\lambda_1=\min\sigma(A)$ be the smallest eigenvalue of $A$.
By Proposition~\ref{polezeroprop2}~(i) and \eqref{msym} the
 isolated eigenvalues of the Sturm-Liouville operator $A$ coincide with
the poles and zeros of the function $m_+$.
Hence $\lambda_1$ is either a pole or a zero of $m_+$. Since $m_+$
is a Nevanlinna function the poles and zeros
of $m_+$ in $(\alpha,\beta)$ alternate. Therefore
one of the following four cases occurs if $n:=n_A(\alpha,\beta)<\infty$:
\begin{itemize}
 \item [{\rm (i)}] $n$ is even and $\lambda_1$ is a pole of $m_+$;
  \item [{\rm (ii)}] $n$ is even and $\lambda_1$ is a zero of $m_+$;
 \item [{\rm (iii)}] $n$ is odd and $\lambda_1$ is a pole of $m_+$;
 \item [{\rm (iv)}] $n$ is odd and $\lambda_1$ is a zero of $m_+$.
\end{itemize}
In case (i) the function $m_+$ has $\tfrac{n}{2}$ poles and $\tfrac{n}{2}$ zeros in $(\alpha,\beta)$. Moreover,
the largest eigenvalue $\lambda_n$ of $A$ in $(\alpha,\beta)$ is a zero of $m_+$ and hence $m_+$ is positive on $(-\infty,\lambda_1)\cup(\lambda_n,\beta)$. The function $M$ has $\tfrac{n}{2}$ poles in $[\lambda_1,\lambda_{n-1}]$
and it follows from Corollary~\ref{corcor} that there are $\tfrac{n}{2}-1$ zeros of $M$ in $(\lambda_1,\lambda_{n-1})$. Since $m_+$ is positive on $(-\infty,\lambda_1)$ and $m_+(\lambda_n)=0$
it follows that $M$ has also one zero in $(\lambda_{n-1},\lambda_n)$, and is positive on $(\alpha,\lambda_1)$
and $[\lambda_n,\beta)$. Now Proposition~\ref{polezeroprop}~(i) implies
$n_{JA}(\alpha,\beta)=\tfrac{n}{2}=\tfrac{1}{2}n_A(\alpha,\beta)$
and by symmetry also $n_{JA}(-\beta,-\alpha)=\tfrac{n}{2}=\tfrac{1}{2}n_A(\alpha,\beta)$.
The simple modifications of this argument for case (ii) are left to the reader.

In case (iii) the function $m_+$ has $\tfrac{1}{2}(n+1)$ poles and $\tfrac{1}{2}(n-1)$ zeros in $(\alpha,\beta)$. Moreover,
$m_+$ is positive on $(-\infty,\lambda_1)$ and since the largest eigenvalue $\lambda_n$ of $A$ in $(\alpha,\beta)$ is a pole $m_+$ is negative on $(\lambda_n,\beta)$. The function $M$ has $\tfrac{1}{2}(n+1)$ poles in $[\lambda_1,\lambda_n]$
and it follows from Corollary~\ref{corcor} that there are $\tfrac{1}{2}(n-1)$ zeros of $M$ in $(\lambda_1,\lambda_n)$. Furthermore, since $m_+$ is positive on $(-\infty,\lambda_1)$ and negative on $(\lambda_n,\beta)$ there may be one more zero of $M$ in $(\lambda_n,\beta)$. Now Proposition~\ref{polezeroprop}~(i) implies $n_{JA}(\alpha,\beta)=\tfrac{1}{2}(n\pm 1)=\tfrac{1}{2}(n_A(\alpha,\beta)\pm 1)$ and by symmetry also $n_{JA}(-\beta,-\alpha)=\tfrac{1}{2}(n\pm 1)=\tfrac{1}{2}(n_A(\alpha,\beta)\pm 1)$.
The simple modifications of this argument for case (iv) are left to the reader.
Relation \eqref{DaSchaust} is proved.

From \eqref{DaSchaust} it follows also that for $\beta$
in the essential spectrum of $A$ the quantity  $n_{A}(\alpha,\beta)$ is finite
if and only if the quantities
$n_{JA}(\alpha,\beta)$ and $n_{JA}(-\beta,-\alpha)$ are finite.
\end{proof}

The next proposition on the interlacing properties of the eigenvalues of $JA$ with respect
to the eigenvalues of $A$ can be shown with the same methods as Theorem~\ref{mainthm2}.
If $(a,b)$ is a gap in $\sigma_\ess(A)$ we denote by $(\lambda_k)$
the eigenvalues of $A$ in increasing order, where $k=1,\dots,n_A(a,b)$ if $n_A(a,b)$ is finite,
$k\in\dN$ ($k\in-\dN$) if the eigenvalues accumulate to $b$ ($a$, respectively), and
$k\in\dZ$ if both endpoints $a$ and $b$ are accumulation
points of eigenvalues of $A$.

\begin{proposition}
Assume that conditions (I), (II), and (III) hold for the Sturm-Liouville operator $A$
and let $JA$ be the corresponding indefinite Sturm-Liouville differential
operator. Let $(a,b)\cap\sigma_\ess(A)=\emptyset$ and denote by $(\lambda_k)$ the eigenvalues of $A$
in $(a,b)$ in increasing order. Then exactly one of the following statements hold:
\begin{itemize}
\item [{\rm (i)}] Each interval $(\lambda_{2k-1},\lambda_{2k})$ contains exactly one
eigenvalue of $JA$ and each interval $[\lambda_{2k},\lambda_{2k+1}]$ belongs to $\rho(JA)$;
\item [{\rm (ii)}] Each interval $(\lambda_{2k},\lambda_{2k+1})$ contains exactly one
eigenvalue of $JA$ and each interval $[\lambda_{2k-1},\lambda_{2k}]$ belongs to $\rho(JA)$.
\end{itemize}
Furthermore, in the case $a<\lambda_1=\min\sigma(A)<b\leq\min\sigma_\ess(A)$ statement (i) holds,
that is, for the positive eigenvalues $\lambda_k(JA)$ of $JA$ ordered in an increasing way we have
\begin{equation*}
\lambda_k (JA) \in \left(\lambda_{2k-1}, \lambda_{2k}\right),\quad k=1,2,\dots
\end{equation*}
\end{proposition}

\section{Examples}

In this section some applications and examples illustrating
the results in the previous section are presented. We start with
a variant of Kneser's classical oscillation result in the context
of indefinite Sturm-Liouville operators. As a second application a periodic
problem is treated and in a third explicit example the number of eigenvalues
of the indefinite operator is computed for a particularly simple potential.

\subsection{Kneser's result for left-definite Sturm-Liouville operators}\label{Kneserli}
In this first example accumulation of the eigenvalues of $JA$ to the essential spectrum
is studied with the help of Kneser's classical result from \cite{K11},
see also \cite[Corollary XIII.7.57]{DS63} and \cite{GST96,GUe,KT,KT9,W87} for possible generalizations.
Here, for simplicity, let $r(x)=\sgn(x)$, $p(x)=1$, and assume that $q>0$ admits the
positive limits
\begin{equation*}
0<q_\infty:=\lim_{x\rightarrow +\infty} q(x)=\lim_{x\rightarrow -\infty} q(x).
\end{equation*}
Clearly, condition (I) holds with $c=0$ and by well-known results (see, e.g. \cite[Theorem 6.3]{W87})
the corresponding maximal Sturm-Liouville operator $Af=-f^{\prime\prime}+qf$, $f\in\sD$,
satisfies condition (II). Here we have $\sigma_\ess(B_\pm)=[q_{\infty},\infty)$ and therefore
\begin{equation*}
\sigma_\ess(A)=[q_{\infty},\infty).
\end{equation*}
By Propositions~\ref{indefprop1} and \ref{indefprop2} the essential spectrum of
the $J$-selfadjoint indefinite Sturm-Liouville operator $JA f= \sgn(-f^{\prime\prime}+qf)$, $f\in\sD$,
is then given by
\begin{equation*}
\sigma_\ess(JA)=(-\infty,-q_{\infty}]\cup[q_\infty,\infty).
\end{equation*}

Let us now make use of Kneser's criterion: If
\begin{equation}\label{Kneser}
\limsup_{x\rightarrow \infty}\, x^2 (q(x)-q_\infty)<-\frac{1}{4} \quad \text{or}\quad
\limsup_{x\rightarrow -\infty}\, x^2 (q(x)-q_{\infty})<-\frac{1}{4}
\end{equation}
holds, then there are infinitely many eigenvalues of $B_+$ or $B_-$, respectively, below their essential spectrum
and hence also the eigenvalues of $A$ accumulate to $\min\sigma_\ess(A)$. By Theorem~\ref{mainthm1}
there are infinitely many eigenvalues of $JA$ in the corresponding gap $(-q_{\infty},q_\infty)$ in $\sigma_\ess(JA)$.
In the present situation it follows also that the eigenvalues of $JA$ in $(-q_{\infty},q_\infty)$
accumulate to $q_\infty$ ($-q_{\infty}$) if the first (second, respectively)
condition in \eqref{Kneser} holds.

Similarly, if instead of \eqref{Kneser} we have
\begin{equation*}
\liminf_{x\rightarrow \infty}\, x^2 (q(x)-q_\infty)>-\frac{1}{4} \quad \text{and}\quad
\liminf_{x\rightarrow -\infty}\, x^2 (q(x)-q_{\infty})>-\frac{1}{4},
\end{equation*}
then there are only finitely many eigenvalues of $B_+$ and $B_-$ below their essential spectrum
and hence there are also only finitely many eigenvalues of $A$ below $\min\sigma_\ess(A)$.
In this situation Theorem~\ref{mainthm1} implies that $JA$ has only finitely many
eigenvalues in the corresponding gap around zero and their total number in $(-q_{\infty},q_\infty)$ differs at most by three
of the number of eigenvalues of $A$ below $q_{\infty}=\min\sigma_\ess(A)$.

\subsection{Periodic operators} Suppose that the coefficients $\vert r\vert$, $p$ and $q$ of the
definite Sturm-Liouville expression $\ell$ are $\gamma$-periodic for some $\gamma>0$ and
assume that $\text{essinf}\, q/\vert r\vert$ is positive as well as $r$ satisfies condition (I). Then condition (II) is satisfied
for the corresponding maximal operator $A$ in $L^2_{\vert r\vert}(\dR)$. Furthermore, let $\lambda_1<\lambda_2 \leq\lambda_3 \leq \dots$  be
the eigenvalues of the selfadjoint operator associated with $\ell$ restricted to functions in $L^2_{\vert r\vert}(0,\gamma)$
with the boundary conditions
\begin{equation*}
\begin{pmatrix} f(0)\\ (pf^\prime)(0) \end{pmatrix}=\begin{pmatrix} f(\gamma)\\ (pf^\prime)(\gamma) \end{pmatrix}
\end{equation*}
and let $\mu_1\leq \mu_2\leq\mu_3\leq \cdots$ be
the eigenvalues of the selfadjoint operator associated with $\ell$ restricted to functions in $L^2_{\vert r\vert}(0,\gamma)$
with the boundary conditions
\begin{equation*}
\begin{pmatrix} f(0)\\ (pf^\prime)(0) \end{pmatrix}=-\begin{pmatrix} f(\gamma)\\ (pf^\prime)(\gamma)
\end{pmatrix}.
\end{equation*}
Then $0<\lambda_1<\mu_1\leq\mu_2 <\lambda_2\leq\lambda_3<\mu_3\dots$ and it is well-known that
\begin{equation*}
 \sigma(A)=\sigma_\ess(A)=[\lambda_1,\mu_1]\cup[\mu_2,\lambda_2]\cup[\lambda_3,\mu_3]\dots
\end{equation*}
holds, see, e.g., \cite[$\S\,$12]{W87}. Here it follows that also $\sigma_\ess(B_+)=\sigma_\ess(B_-)=\sigma_\ess(A)$ holds, and therefore
by Proposition~\ref{indefprop2}
the essential spectrum $\sigma_\ess(JA)$ of $JA$ has a band structure, is symmetric with respect to $0$
and is given by
\begin{equation*}
\dots [-\mu_3,-\lambda_3]\cup[-\lambda_2,-\mu_2]\cup [-\mu_1,-\lambda_1]\cup [\lambda_1,\mu_1]\cup[\mu_2,\lambda_2]\cup[\lambda_3,\mu_3]\dots.
\end{equation*}
Since $A$ has no eigenvalues in
the (possible) gaps $(\mu_1,\mu_2)$, $(\lambda_2,\lambda_3)$, $(\mu_3,\mu_4),\dots$,
of $\sigma_\ess(A)$ we conclude from
Theorem~\ref{mainthm1} that each of the sets
\begin{equation*}
(-\mu_2,-\mu_1)\cup(\mu_1,\mu_2),\quad
(-\lambda_3,-\lambda_2)\cup(\lambda_2,\lambda_3),\quad (-\mu_4,-\mu_3)\cup(\mu_3,\mu_4),\dots
\end{equation*}
contains at most $3$ eigenvalues of the indefinite Sturm-Liouville operator $JA$.
Note that by Proposition~\ref{gapcor} we have
$(-\lambda_1,\lambda_1)\subset\rho(JA)$. Furthermore, if the coefficients $r$, $p$, and $q$ satisfy
the symmetry condition (III), then Corollary~\ref{cor234} implies that in each of the (possible) gaps
$$
\dots (-\mu_4,-\mu_3),\,\,(-\lambda_3,-\lambda_2),\,\,(-\mu_2,-\mu_1),\,\,(\mu_1,\mu_2),\,\,(\lambda_2,\lambda_3),\,\,
(\mu_3,\mu_4),\dots
$$
of
$\sigma_\ess(JA)$ there is at most one eigenvalue.

\subsection{A solvable problem with a hyperbolic cosine potential.}
As an explicit example consider the situation $r(x)=\sgn x$, $p(x)=1$ and
$$
q(x) =
(\kappa +1)^2-\frac{\kappa(\kappa +1)}{{\cosh}^2(x)}\qquad\text{for some}\quad\kappa\in\dN.
$$
Obviously, conditions (I) and (III) are satisfied. Moreover, $q(x)\geq\kappa+1$ and
$\lim_{|x|\to\infty} q(x) =(\kappa +1)^2$ imply that for the
corresponding maximal operator $A$ we have $\min \sigma(A) \geq \kappa+1$ and $\sigma_\ess (A) =[(\kappa +1)^2,\infty)$. In particular, condition (II) is also fulfilled.
It is known (see, e.g.,\ \cite{H92}), that the operator $A$
has precisely $\kappa$ eigenvalues in the interval $(\kappa +1, (\kappa +1)^2)$.
Therefore, the essential spectrum of the corresponding indefinite Sturm-Liouville operator $JA$
is given by $$\sigma_\ess(JA)=(-\infty,-(\kappa+1)^2]\cup[(\kappa+1)^2,\infty)$$ and
by Theorem \ref{mainthm2} the operator $JA$
has $\tfrac{\kappa}{2}$ eigenvalues in the interval $(\kappa +1, (\kappa +1)^2)$ if $\kappa$ is even and $\tfrac{\kappa\pm 1}{2}$ eigenvalues if $\kappa$ is odd.
The same holds for the interval $(-(\kappa +1)^2, -(\kappa +1))$;
cf.\ Theorem \ref{mainthm2}. Note that by Proposition~\ref{gapcor} $\kappa+1$ and $-(\kappa+1)$ are no
eigenvalues of $JA$.

\subsection*{Acknowledgement}
The authors thank Gerald Teschl for fruitful remarks.

\end{document}